\documentclass{amsart}
\usepackage{amssymb,amsmath,amscd,amsthm}
\usepackage{amsfonts,epsfig,latexsym,graphicx,amssymb}
\usepackage{color}
\usepackage{enumerate}
\usepackage{todonotes}

\usepackage{hyperref}

\date{\today}

\newcommand{\Z}{{\mathbb Z}}
\newcommand{\R}{{\mathbb R}}

\newcommand{\N}{{\mathbb N}}

\renewcommand{\P}{{\mathbb P}}

\newcommand{\E}{{\mathbb E}}

\newcommand{\Prob}{\mathop{\mathbb{P}}\nolimits}

\newcommand{\bK}{\mathbf{K}}
\newcommand{\mM}{\mathcal{M}}
\newcommand{\mF}{\mathcal{F}}

\newcommand{\dist}{\mathop{\mathrm{dist}}\nolimits}
\newcommand{\Leb}{\mathop{\mathrm{Leb}}\nolimits}
\newcommand{\Ind}{1\!\!\mathrm{I}}

\newcommand{\id}{\mathrm{id}}

\newtheorem{theorem}{Theorem}[section]
\newtheorem*{theorem*}{Theorem}
\newtheorem{lemma}[theorem]{Lemma}
\newtheorem{prop}[theorem]{Proposition}
\newtheorem{coro}[theorem]{Corollary}

\newtheorem{example}[theorem]{Example}
\theoremstyle{definition}
\newtheorem{remark}[theorem]{Remark}

\newtheorem{defi}[theorem]{Definition}

\newcommand{\eps}{{\varepsilon}}
\newcommand{\mgr}{\mu}  
\newcommand{\msp}{\nu}  

\title[Non-stationary Ergodic Theorem]{Non-stationary version of Ergodic Theorem \\ for random dynamical systems}

\author[A.\ Gorodetski]{Anton Gorodetski}

\address{Department of Mathematics, University of California, Irvine, CA~92697, USA}

\email{asgor@uci.edu}

\thanks{A.\ G.\ was supported in part by NSF grant DMS--1855541. Part of the work was done while A.\ G.\ was in residence at Institute Mittag-Leffler in Djursholm, Sweden, during the Spring semester of 2023, and was supported by the Swedish Research Council under grant  no. 2016-06596. } 

\author[V. Kleptsyn]{Victor Kleptsyn}

\address{CNRS, Institute of Mathematical Research of Rennes, IRMAR, UMR 6625 du CNRS}

\email{victor.kleptsyn@univ-rennes1.fr}

\thanks{V.K. was supported in part by ANR Gromeov (ANR-19-CE40-0007) and by Centre Henri Lebesgue (ANR-11-LABX-0020-01)}

\begin{document}

\maketitle

\begin{center}
{\it To our teacher, Professor Yulij Sergeevich Ilyashenko}
\end{center}

\

\begin{abstract}
We prove a version of pointwise Ergodic Theorem for non-stationary random dynamical systems. Also, we discuss two specific examples where the result is applicable: non-stationary iterated function systems and non-stationary random matrix products.
\end{abstract}

\section{Introduction}\label{s.intro}

Birkhoff Ergodic Theorem is one of the key tools of the theory of dynamical systems. In the context of topological dynamics it claims that for a continuous dynamical system $f:X\to X$ on a metric compact $X$ and an $f$-invariant measure $\msp$ on $X$, given a ``test function'' $\varphi\in C(X)$,  the time averages
\[
\frac{1}{n} \sum_{j=0}^{n-1}\varphi(f^j (x))
\]
converge for $\msp$-a.e. $x\in X$, and if the measure $\msp$ is ergodic, the limit almost everywhere equals the space average $\int_X \varphi \, d\msp$.

It admits a natural generalization for the case of random dynamical systems. Namely, let $X$ be a compact metric space, and $\mgr$ be a probability measure on the space of continuous maps $C(X, X)$. The iterations of the corresponding random dynamical system are the sequences of compositions
\[
f_1,\  f_2\circ f_1,\  \dots, \ f_n\circ \dots \circ f_1,\  \dots,
\]
where $f_i:X\to X$ are chosen randomly and independently w.r.t. the measure~$\mgr$.

An analogue of the notion of an invariant measure is the one of a stationary measure. Namely, a $\mgr$-\emph{stationary} measure $\msp$ is a probability measure on~$X$ such that $\mgr * \msp=\msp$, where the convolution~$\mgr * \msp$ is the law of $f(x)$ for independent $f$ chosen w.r.t.~$\mgr$ and $x$ w.r.t.~$\msp$. Such a measure always exists (this can be shown via an analogue of the Krylov-Bogolyubov averaging procedure), though is not unique in general (in the same way as an invariant measure is not unique for a classical deterministic dynamical system in general).

The following theorem corresponds to the Birkhoff Ergodic Theorem for the stationary random dynamics:
\begin{theorem}[Random Birkhoff Ergodic Theorem]
For any ergodic stationary measure $\nu$ on $X$, for any $\varphi\in C(X, \mathbb{R})$, for $\nu$-a.e. $x\in X$, $\mu^{\mathbb{N}}$-almost surely one has
$$
\frac{1}{n}\sum_{k=0}^{n-1}\varphi(f_k\circ\ldots\circ f_1(x))\to \int_X\varphi(x)d\nu(x),
$$
where $f_1, f_2, \ldots$ are chosen randomly and independently, with respect to the distribution $\mu$.
\end{theorem}

However, this generalisation heavily relies on the fact that $f_n$'s are identically distributed.
What can be said if $f_n$'s are still chosen independently, but each $f_n$ is distributed w.r.t. a different measure~$\mgr_n$? This can be represented symbolically by the following diagram:

\[
{X \xrightarrow[f_1]{\mgr_1} X \xrightarrow[f_2]{\mgr_2} X \rightarrow \dots \rightarrow X \xrightarrow[f_n]{\mgr_n} X \rightarrow \dots}
\]

At first glance, many notions and tools disappear, starting with the one of the stationary measure: the measures $\mu_n$ on different steps may differ, and the system of equations $\msp=\mgr_n * \msp$ for all $n$ is usually incompatible. Thus even finding a proper statement, not to mention proving it, is non-evident.

Meanwhile, such a setting naturally arises in some situations. For example, the study of one-dimensional Anderson Localization problem~\cite{GK3} in presence of a background potential leads to the study of products of independent non-identically distributed matrices, the non-stationary Furstenberg Theorem~\cite{GK2}, that is within the frames of non-stationary ergodic theory. It also motivated the study of H\"older regularity
of averaged images of a given initial measure in~\cite[Theorem 2.8]{GKM}.

The goal of the present paper is to discuss a paradigm for ergodic theorems in a non-stationary setting, presenting both (counter)-examples and a restricted version of an ergodic theorem, see Theorem~\ref{t.main} below. The latter is only a ``proof-of-concept''; it definitely can be vastly generalized, and we hope it soon will be: our goal is to motivate the farther research in this direction.

\section{Preliminaries}

Throughout this paper, we set $X$ to be a compact metric space, and assume that a ``test function''
$\varphi\in C(X)$ is given. Also, for every $n$ we assume that a measure $\mgr_n$ on $C(X,X)$ is given (if the dynamics is assumed to be invertible, one can ask instead for the measure on the set of homeomorphisms of~$X$).
We then denote
\begin{equation}\label{eq:Prob}
{\Prob}
:=\prod_{n=1}^\infty \mgr_n,
\end{equation}
and our goal is to describe time averages along ${\Prob}$-almost every sequence of iterations.

Let us denote by $\mathcal{M}=\mathcal{M}(X)$ the space of Borel probability measures on the compact metric space~$X$. We consider it to be equipped with the Wasserstein distance (that is one of the ways to metrize the weak-* topology in the space of probability measures):
\begin{defi}
    Let $\msp_1, \msp_2\in \mathcal{M}$ be two probability measures on $X$. The \emph{Wasserstein distance} between $\msp_1$ and $\msp_2$ is defined as
    \begin{equation}\label{eq:W-def}
        \text{dist}_{\mathcal{M}}(\msp_1, \msp_2) = \inf_{\gamma} \iint_{X\times X} d_X(x, y) \, d \gamma(x, y),
    \end{equation}
where the infimum is taken over all probability measures $\gamma$ on $X \times X$ with the marginals $\msp_1$ and $\msp_2$, that is, $(\pi_j)_* \gamma = \msp_j$, $j=1,2$, where $\pi_{1,2}:X\times X\to X$ is the projection on the first and second factor respectively.
\end{defi}

Next, note that for a random orbit
\[
(x_n)_{n\in\N}, \quad x_n=f_n(x_{n-1})
\]
the non-stationarity of the dynamics implies that the law of $x_n$ can vary quite strongly, thus the same applies to the law of $\varphi(x_n)$ and to its expectation. Thus, we should not expect the time averages $\frac{1}{n} \sum_{k=0}^{n-1} \varphi(x_k)$ to converge. Instead, we should expect these to have a deterministic behaviour, similarly to the non-stationary Law of Large Numbers. Namely, under reasonable assumptions for independent, but not identically distributed random variables $\xi_n$ one has
\[
\frac{1}{n} \left| \sum_{k=1}^{n} \xi_k  -  \sum_{k=1}^{n} \E \xi_k \right| \to 0, \quad n\to\infty.
\]
As the Birkhoff Erogic Theorem generalises the Law of Large Numbers, we should expect
\begin{equation}\label{eq:limit}
\frac{1}{n} \left| \sum_{k=1}^{n} \varphi(x_k)  -  \sum_{k=1}^{n} \int_X \varphi(x) d\nu_k(x) \right| \to 0, \quad n\to\infty
\end{equation}
for some reasonable non-random measures~$\nu_k$. Indeed, this turns out to be the case: see Theorem~\ref{t.main} below.

\section{Statement of the main result}

Our main result will be stated under the following assumption. This assumption, in a sense, generalizes (and slightly strengthens) the \emph{unique} ergodicity. The latter requires the stationary measure to be unique, so this setting is more restrictive than ``simple'' ergodicity; however, in this setting such an assumption seems to be appropriate, see the discussion in Section~\ref{s:counter} below, and Example~\ref{ex:no-average} therein.

\vspace{3mm}

{\bf Standing Assumption:} {\it We will say that a sequence of distributions $\mgr_1, \mgr_2, \mgr_3, \ldots$ on $C(X, X)$ satisfies the Standing Assumption if for any $\delta>0$ there exists $m\in \mathbb{N}$ such that the images of any two initial measures after averaging over $m$ random steps after any initial moment $n$ become $\delta$-close to each other:
$$
\forall \msp, \msp'\in \mathcal{M}, \quad \forall n\in \mathbb{N}, \quad  \dist_{\mathcal{M}} (\mgr_{n+m}*\ldots *\mgr_{n+1} *\msp, \mgr_{n+m}*\ldots *\mgr_{n+1} *\msp') < \delta.
$$
}

Note that if after a given number $m$ of steps the averaged images of any two initial measures $\msp, \msp'$ are close to each other, one can pick any of these images and state that all the others are close to it. So one can pick any measure $\msp_0$ (for instance, a Dirac one at a given initial point $x_0$) and define
\begin{equation}\label{eq:nu-def}
\msp_{n}:=\mgr_n*\msp_{n-1}, \quad n=1,2,\dots.
\end{equation}
Then, for any given measure $\msp$, the convolution $\mgr_{n+m}*\ldots *\mgr_{n+1} *\msp$ is close to $\msp_{n+m}$, and this motivates the appearance of measures $\msp_n$ in~\eqref{eq:limit}.

Here is the main result of this paper:

\begin{theorem}\label{t.main}
Suppose the sequence of distributions $\mgr_1, \mgr_2, \mgr_3, \ldots$ satisfies the Standing Assumption above. Given any Borel probability measure $\msp_0$ on $X$, define
$$
\msp_{n}:=\mgr_n*\msp_{n-1}, \quad n=1,2,\dots.
$$

Then for any $\varphi\in C(X, \mathbb{R})$ and any $x\in X$, almost surely
$$
\frac{1}{n}\left|\sum_{k=1}^n \varphi(f_k\circ \ldots \circ f_1(x))-\sum_{k=1}^n \int_X \varphi \, d\nu_k \right|\to 0 \ \ \text{\rm as}\ \ \ n\to \infty,
$$
where the measures $\msp_n$ are defined by~\eqref{eq:nu-def}.

Moreover, an analogue of the Large Deviations Theorem holds. Namely, for any $\eps>0$ there exist $C,\delta>0$ such that for any $x\in X$
\begin{equation}\label{eq:LD}
\forall n\in \mathbb{N}, \quad \Prob\left(
\frac{1}{n}\left|\sum_{k=1}^{n} \varphi(f_k\circ \ldots \circ f_1(x))-\sum_{k=1}^n \int_X \varphi \, d\nu_k \right|
> \eps
\right) < C \exp(-\delta n).
\end{equation}
\end{theorem}

\begin{remark}
While the Standing Assumption above might seem quite restrictive, it does hold in a natural way in some interesting cases. In the next section we prove that it holds in two settings, namely, for random dynamical systems defined by contractions, also known as random iterated function systems, and for random dynamical systems on a projective space generated by random matrix products.
\end{remark}

\begin{remark}
In the stationary case, when $\mu_i=\mu$ for all $i\in \mathbb{N}$, Standing Assumption implies that the random dynamical system defined by the distribution $\mu$ is uniquely ergodic, i.e. has a unique stationary measure. Notice, however, that the Standing Assumption is actually stronger (i.e. more restrictive) than the unique ergodicity. Namely, unique ergodicity is equivalent to the statement that \emph{time averages} of convoluted images of any initial measure~$\msp$ converge to the unique stationary measure~$\overline{\msp}$ uniformly in~$\msp$:
\[
\frac{1}{m} \sum_{j=0}^m \mgr^{(*j)}*\msp \to \overline{\msp}, \quad m\to \infty,
\]
where $\mgr^{(*j)}$ is $j$-th convolution power. 
Example~\ref{ex:no-average} explains why the Standing Assumption in the non-stationary setting should indeed be stronger.
\end{remark}

\begin{remark}
We would like to point out that there are different settings where non-stationary versions of ergodic theorems could be and were established. For example, the setting where a sequence of maps that preserve the same fixed measure was considered in \cite{BB}. The case when all maps are non-singular with respect to the Lebesgue measure, so the transfer operator technics are applicable, was studied in \cite{CR, HNTV}. Sometimes a term {\it sequential dynamical system} is used in these contexts. We emphasize that both the setting and the underlying mechanisms there are different from ours.
\end{remark}

\section{Examples of applicability}

We will give two examples of non-stationary random dynamical systems for which Standing Assumption is satisfied and, hence, Theorem \ref{t.main} is applicable.
\subsection{Random Iterated Function Systems}

An iterated function system is defined by a finite collection of contractions of a complete metric space, usually $\mathbb{R}^d$ or a compact subset of $\mathbb{R}^d$. In many cases it leads to a fractal attractor. For example, the standard Cantor set can be defined via two contractions of the unite interval, $x\mapsto \frac{x}{3}$ and $x\mapsto \frac{x}{3}+\frac{2}{3}$. Many self-similar fractals can be generated this way. If probabilities are assigned to each of the contractions, one can consider a corresponding random dynamical system. It is known that it must have unique stationary measure \cite{Hu, Sz}. Properties of this stationary measure were a subject of intense studies. There is a vast amount of literature  on exact-dimensionality and fractal dimension of this measure; we will only mention a recent survey \cite{FS}. H\"older regularity of the this measure was studied in \cite[Proposition 2.2]{FL} and  \cite[Section 1.2]{GKM}. For a survey of the theory of Bernoulli convolutions that also fall into this setting see \cite{PSS} or a more recent \cite{Va}.

One can consider a non-stationary random dynamical system generated by contractions, where a different distribution on the space of contraction can be chosen on each step. This setting was considered, for example, in \cite{GKM}; it was shown there that under some mild and natural conditions the averaged iterates of any given initial measure must converge to the space of H\"older regular measures exponentially fast.

One can easily see that in this setting the Standing Assumption holds, and hence, Theorem \ref{t.main} is applicable:

\begin{prop}\label{p.contr}
Let $X$ be a compact metric space, and $\lambda\in (0,1)$ be a constant. Suppose $\{\mu_i\}_{i\in \mathbb{N}}$ be a sequence of probability distributions in the space of contractions $X\to X$ with Lipschitz constant at most $\lambda$. Then Standing Assumption holds, i.e. for any Borel probability measures $\nu, \nu'$ on $X$ we have
$$
\text{\rm dist}_{\mathcal{M}} (\mu_{n+m}*\ldots *\mu_{n+1} *\nu, \mu_{n+m}*\ldots *\mu_{n+1} *\nu')\to 0\ \ \text{\rm as}\ \ \ m\to \infty,
$$
uniformly in $(\nu, \nu')$ and $n\in \mathbb{N}$.
\end{prop}
\begin{proof}
Indeed, for any $\mu_i$, any map $f\in \text{supp}\, \mu_i$ is a contraction with Lipschitz constant at most $\lambda$. Hence for any $\nu_1, \nu_2\in \mathcal{M}$ we have
$$
\text{\rm dist}_{\mathcal{M}} (\mu*\nu_1, \mu*\nu_2)\le \lambda \text{\rm dist}_{\mathcal{M}} (\nu_1, \nu_2),
$$
and Proposition \ref{p.contr} follows.
\end{proof}

\subsection{Random Matrix Products}

Let us now show that under some non-degeneracy assumptions a non-stationary random dynamical system defined by projective maps must satisfy Standing Assumption.

Suppose $\bK$ is a compact set in the space of probability distributions on $SL(2, \mathbb{R})$  such that for any $\mu\in \bK$ {\it the measures condition} is satisfied (we use terminology from~\cite{GKM}), i.e. there are no probability measures $\nu_1, \nu_2$ on $\mathbb{RP}^1$ such that $(f_A)_*(\nu_1)=\nu_2$ for all  $A\in \text{\rm supp}\,\mu$. Slightly abusing the notation, we will treat $\mu$ also as a measure on the space of projective maps $f_A:\mathbb{RP}^1\to \mathbb{RP}^1$.

\begin{prop}\label{p.proj}
Suppose the measure condition holds for each $\mu\in \bK$. Then for any sequence $\{\mu_i\}\in \bK^{\mathbb{N}}$ and any probability measures $\nu, \nu'\in \mathcal{M}$ we have:
$$
\text{\rm dist}_{\mathcal{M}} (\mu_{n+m}*\ldots *\mu_{n+1} *\nu, \mu_{n+m}*\ldots *\mu_{n+1} *\nu')\to 0\ \ \text{\rm as}\ \ \ m\to \infty,
$$
uniformly in $(\nu, \nu')$ and $n\in \mathbb{N}$.
\end{prop}

{We will deduce this proposition from the following statement, that is also of independent interest:}
\begin{prop}\label{p.nF}
{Suppose the measure condition holds for each $\mu\in \bK$. Then for any $\eps>0$ there exists $m$ such that for any measures $\mgr_1,\dots,\mgr_m\in\bK$ and any points $x,y\in\R\P^1$ one has
\[
(\mgr_m*\dots*\mgr_1)\{F \mid d(F(x),F(y))<\eps \} > 1-\eps,
\]
 where the composition $F=f_m\circ\dots\circ f_1$ is formed by independent random maps $f_i$ distributed w.r.t. $\mgr_i$. In other words, the random composition $F$ brings $x$ and $y$ to the distance less than~$\eps$ with probability at least~$1-\eps$.
}
\end{prop}

This deduction is indeed immediate:

\begin{proof}[Proof of Proposition~\ref{p.proj}]
Note first that
\begin{multline*}
\text{\rm dist}_{\mathcal{M}} (\mgr_{n+m}*\ldots *\mgr_{n+1} *\msp, \mgr_{n+m}*\ldots *\mgr_{n+1} *\msp')\le
\\
\iint_{\mathbb{RP}^1} \dist_{\mathcal{M}} (\mgr_{n+m}*\ldots *\mgr_{n+1} *\delta_x, \mgr_{n+m}*\ldots *\mgr_{n+1} *\delta_y)\, d\nu(x) d\nu'(y);
\end{multline*}
this follows from the convexity of the transport distance: mixing the couplings provides a coupling for the averaged measures.

At the same time, we have
\begin{multline}\label{eq:dF}
\text{\rm dist}_{\mathcal{M}} (\mgr_{n+m}*\ldots *\mgr_{n+1} *\delta_x,
\mgr_{n+m}*\ldots *\mgr_{n+1} * \delta_y)\le \\
\le
\int d(F(x), F(y)) \, d(\mgr_{n+m}*\ldots *\mgr_{n+1})(F).
\end{multline}
Take an arbitrary $\eps>0$; once $m$ is taken sufficiently large so that the conclusion of Proposition~\ref{p.nF} holds, the right hand side of~\eqref{eq:dF} does not exceed
\begin{multline*}
\eps \cdot  (\mgr_{n+m}*\dots*\mgr_{n+1})\{F \mid d(F(x),F(y))<\eps \} +
\\+ 1 \cdot  (\mgr_{n+m}*\dots*\mgr_{n+1})\{F \mid d(F(x),F(y))\ge \eps \} \le
\eps \cdot 1 + 1 \cdot \eps = 2\eps.
\end{multline*}
As $\eps>0$ is arbitrary, we get the desired convergence that is uniform in~$n$, $\nu$, and $\nu'$.
\end{proof}

Let us now return to Proposition~\ref{p.nF}. Before proceeding to the proof, note that if the dynamics was stationary and if the expectation of $\log \|A\|$ was bounded, the conclusion of the proposition would be implied by the famous Furstenberg Theorem. This motivates to approach it via a non-stationary version of Furstenberg Theorem. Such versions were studied recently in~\cite{GK2} and~\cite{Gold}.

\begin{proof}[Proof of Proposition~\ref{p.nF}]
Assume that the expectation of $\|A\|$ is bounded uniformly in $\mgr\in\bK$:
\begin{equation}\label{e.7}
\exists C_\bK: \quad \forall \mgr\in \bK \quad \int \|A\| \, d\mgr(A) < C_\bK.
\end{equation}
Then we are in a situation when Theorems~1.1 and~1.4 from~\cite{GK2} are applicable. Namely, denote
\[
L_n:=\E \log \|A_n\dots A_1\|.
\]
Theorem~1.1 from~\cite{GK2} then implies a linear lower bound for $L_n$: there exists $\lambda_{\bK}>0$ such that
\[
L_n \ge n \lambda_{\bK};
\]
moreover, $\lambda_{\bK}$ does not depend on the individual choices of $\mgr_i$, but only on the compact~$\bK$ itself.

Now, Theorem~1.4 from~\cite{GK2} states that a Large Deviations type estimate holds: for any fixed $\delta>0$, there exists $\delta'$ such that for any $v_0\in\R^2$, $|v_0|=1$, and any sufficiently large $n$ (that is, for all $n$ greater than some $n_0$) one has
\[
\Prob(\left|\log |A_n\dots A_1 v_0| - L_n\right| > n\delta ) < e^{-n \delta'}.
\]
Moreover (see~\cite[{Remark~1.4}]{GK2}), the constants $\delta'$ and $n_0$ can be chosen uniformly in all possible sequences of~$\mgr_i\in\bK$.

To establish the conclusion of the proposition in this particular case, i.e. when (\ref{e.7}) holds,  take $\delta:=\frac{1}{2}\lambda_{\bK}>0$ and consider the corresponding~$n_0$,~$\delta'$. Then, once the inequality
\begin{equation}\label{eq:gr-A}
\left|\log |A_n\dots A_1 v_0| - L_n\right| <n\delta
\end{equation}
holds, it implies
\[
\log |A_n\dots A_1 v_0| > L_n - n \delta \ge \frac{1}{2} \lambda_{\bK} n.
\]
In particular, for any given $\eps>0$, taking $n_1$ to be the integer part of $\frac{\log 2\eps^{-1}}{\frac{1}{2} \lambda_{\bK} n}$, we have
\[
|A_n\dots A_1 v_0| > 2\eps^{-1}
\]
once $n>n_1$ and~\eqref{eq:gr-A} holds. In turn, this inequality implies that the angle between the image $A_n\dots A_1 v_0$ and the image of the most expanded unit vector does not exceed $\frac{\eps}{2}$. Hence, having it for two vectors $v_0,v_0'$ corresponding to two points $x,y\in\R\P^1$ implies that
\begin{equation}\label{eq:dist-F}
\dist (F(x),F(y))<\eps,
\end{equation}
where $F=f_{A_n\dots A_1}$. On the other hand, for any $n>n_2:=\lceil \frac{1}{\delta'} \log 2\eps^{-1} \rceil$ one has
\[
e^{-n\delta'}<\frac{\eps}{2}.
\]

Hence, for $n>\max(n_0,n_1,n_2)$ we get the desired~\eqref{eq:dist-F} with the probability at least
\[
1-2e^{-n\delta'} >1-2\frac{\eps}{2} = 1-\eps
\]
for arbitrary $\mgr_1,\dots,\mgr_n\in \bK$ and $x,y\in\R\P^1$.

Now, let us pass to the general case. The following tool is useful to study random dynamical systems on the circle. Consider an inverse stationary measure, that is, a measure $\msp^-$ such that
\[
\msp^- = \E (f^{-1})_* \msp^-.
\]
Then  the measures of forward random iterations of an interval  form a martingale:
\[
\E \msp^- ( f(J)) = \E ((f^{-1})_* \msp^-) (J) = \msp^- (J).
\]

In the non-stationary setting the notion of \emph{one} stationary measure is no longer applicable. Thus, similarly to the statement of  Theorem~\ref{t.main}, we will replace this notion by a sequence of measures related by convolution. Namely, for every given $n$ and measures $\mgr_1,\dots,\mgr_n\in\bK$, we define
\[
\msp^-_n:=\Leb, \quad \forall i=1,\dots,n \quad \msp^-_{i-1}:= \E_{\mgr_{i}}  (f^{-1})_* \msp^-_i = \mgr_{i}^- *  \msp^-_i,
\]
where the measure $\mgr^-_i$ is the push-forward image of $\mgr_i$ under the inversion map $f\mapsto f^{-1}$. Then we also have the martingale property for the measures of the iterations: for any $i$ and any interval $J$ one has
\begin{equation}
\E_{\mgr_i}\, \msp^-_i (f(J)) =  \msp^-_{i-1} (J),
\end{equation}
and thus
\begin{equation}
\E_{\mgr_n*\dots*\mgr_i}\, \msp^-_n (F(J)) =  \msp^-_{i-1} (J).
\end{equation}

For any given $x,y\in\R\P^1$ consider the sequence of their random iterations: let $f_i$ be chosen independently w.r.t. $\mgr_i$, and denote
\[
x_0:=x, \, y_0:=y, \quad x_{i}=f_i(x_{i-1}), \, y_{i}=f_i(y_{i-1}).
\]
We then have the following lemma.
\begin{lemma}\label{l:inv-meas}
Assume that for some $i$ after $i$ iterations we have $\msp^-_i([x_i,y_i])<\frac{1}{2}\eps^2$ or $\msp^-_i([y_i,x_i])<\frac{1}{2}\eps^2$. Then with the probability at least $1-\frac{\eps}{2}$ we have
\[
\dist(x_n,y_n) < \eps.
\]
\end{lemma}
\begin{proof}
Assume that we are in the first case. Then as
\[
\E ( \msp^-_n([x_n,y_n]) \mid f_1,\dots, f_i) = \msp^-_i([x_i,y_i]) < \eps \cdot \frac{\eps}{2},
\]
due to the Markov inequality with the probability at least $1-\frac{\eps}{2}$ we have
\[
\dist(x_n,y_n)=\Leb([x_n,y_n])=\msp^-_n([x_n,y_n]) < \eps.
\]
The case  $\msp^-_i([y_i,x_i])<\frac{1}{2}\eps^2$ can be treated in the same way.
\end{proof}

Roughly speaking, Lemma~\ref{l:inv-meas} says that if the random iterations of $x$ and $y$ approach each other in the sense of the measure $\msp^-_i$, they most probably will stay sufficiently close to each other in the usual sense at time~$n$, too. The next observation is that measures $\msp^-_i$ are non-atomic, uniformly in both $i$ and possible choices of measures~$\mgr_j$. Namely, we have the following lemma.
\begin{lemma}\label{l:eps-1}
There exists $\eps_1>0$ such that for any $n,i$ and $\mgr_1,\dots,\mgr_n$ and any interval $J$ of length $|J|<\eps_1$ we have $\msp^-_i(J)<\frac{1}{2} \eps^2$.
\end{lemma}
\begin{proof}
Note first that the measures $\mgr^-$ corresponding to the inverse maps also satisfy the measures condition. Indeed, if we had $f^{-1}_* \msp=\msp'$ for $\mgr$-a.e.~$f$ for some $\mgr\in\bK$, then we would also have $f_*\msp'=\msp$, and that would be a contradiction.

Now, the Atoms Dissolving Theorem~1.13 from~\cite{GK2} is applicable, and it states that there exists some $k_0$ such that for any $\mgr_1,\dots,\mgr_{k_0}\in\bK$ and any $x', x''\in\R\P^1$
\begin{equation}\label{eq:a-weights}
(\mgr_1^- * \dots * \mgr^-_{k_0} *\delta_{x'})(\{x''\}) < \frac{1}{2}\eps^2.
\end{equation}
In other words, any (reverse) length $k_0$ convolution makes weights of atoms dissolve so they do not exceed~$\frac{1}{2}\eps^2$. This implies that for some $\eps'$ one has
\begin{equation}\label{eq:J-measure}
(\mgr_1^- * \dots * \mgr^-_{k_0} * \delta_{x'} )(J)<\frac{1}{2}\eps^2 \quad \forall J, \, |J|\le \eps'.
\end{equation}
Indeed, if this statement would be violated for any $\eps'_m=\frac{1}{m}$, then  for some $J_m$, $\mgr_{1,m},\dots,\mgr_{k_0,m}$, we would find a convergent subsequence of these data,
\[
J_{m_j}\to x'', \quad \mgr_{i,m_j}\to \mgr_i \in \bK,
\]
and would get a contradiction with~\eqref{eq:a-weights}. Averaging~\eqref{eq:a-weights} over $x'$ w.r.t. $\msp^-_{i+k_0}$, we obtain the conclusion of the lemma for all $i\le n-k_0$.

Finally, all the measures $\msp^-_{n-k}$ for $k=1,\dots,k_0$ are non-atomic, as by definition $\msp^-_n=\Leb$. Again, this implies that there exists $\eps''$ such that for any $\mgr_{n-k+1}, \dots, \mgr_{n}\in \mathbf{K}$
\begin{equation}
(\mgr_{n-k+1}^- * \dots * \mgr^-_{n} *\msp^-_n )(J)<\frac{1}{2}\eps^2 \quad \forall J, \, |J|\le \eps''.
\end{equation}
Indeed, otherwise we find a convergent subsequence of intervals $J_m$ and measures $\mgr^-_{n-j}\in \bK$, and obtain a contradiction with the non-atomicity mentioned above.

Taking $\eps_1:=\min (\eps',\eps'')$ concludes the proof of Lemma \ref{l:eps-1}.
\end{proof}
Combining the statements of Lemma~\ref{l:eps-1} and Lemma~\ref{l:inv-meas}, we get the following corollary:
\begin{coro}
Assume that for some $i$ after $i$ iterations we have $\dist(x_i,y_i)<\eps_1$. Then with the probability at least $1-\frac{\eps}{2}$ we have
\[
\dist(x_n,y_n) < \eps.
\]
\end{coro}

We are now ready to conclude the proof of Proposition \ref{p.nF}  by considering two different cases. Assume first that there exists $\eps_2>0$ and $k_2$ such that for any $x',y'\in \R\P^1$ and any $\mgr_1,\dots,\mgr_{k_2}\in\bK$ with the probability at least $\eps_2$ the following event holds:
\[
\exists i\le k_2: \quad \dist (x'_i, y'_i) \le \eps_1,
\]
where
\[
x'_0:=x', \, y'_0:=y', \quad x'_{i}=f_i(x'_{i-1}), \, y'_{i}=f_i(y'_{i-1}),
\]
and $f_i$ are random maps  chosen independently w.r.t. $\mgr_i$.

In this case, for any $k_3$ the probability that the iterations of two initial points $x$ and $y$ do not approach each other closer than $\eps_1$ during $k_2 \cdot k_3$ iterations does not exceed $(1-\eps_2)^{k_3}$: we have $k_3$ attempts with at least $\eps_2$ chance of success at each of these. Taking $k_3$ sufficiently large so that
$(1-\eps_2)^{k_3}<\frac{\eps}{2}$ and defining $n:=k_2 \cdot k_3$, we get that with the probability at least $1-\frac{\eps}{2}$ there exists $i\le n$ such that
\[
\dist(x_i,y_i)<\eps_1,
\]
and hence $\dist(x_n,y_n)<\eps$ with the probability at least $(1-\frac{\eps}{2})^2>1-\eps$.

Finally, if such $\eps_2$ and $k_2$ do not exist, we take a sequence of candidates $\eps_{2,(m)}=\frac{1}{m}$, $k_{2,(m)}=m$ and consider the points $x'_{(m)}$, $y'_{(m)}$ and measures $\mgr_{1,(m)}, \dots,\mgr_{m,(m)}\in \bK$, for which the desired statement fails. Extracting a convergent subsequence, we find $x',y'$ and $\mgr_1,\mgr_2,\dots$, for which almost surely the images $x'_i,y'_i$ always stay at the distance at least~$\eps'$.

Now, Atoms Dissolving Theorem~1.13 from~\cite{GK2} together with the compactness argument that we have already applied imply that for some $k_4$, $\eps_3>0$ for any interval $J$ of length $|J|\le \eps_3$ we have
\[
\Prob(x_{k_4}\in J), \Prob(y_{k_4}\in J) \le \frac{1}{3}.
\]
In particular, for any $i>k_4$ the product of matrices $T_{[k_4,i]}:=A_{i}\dots A_{k_4+1}$ cannot have norm higher than $\frac{4}{\eps_3 \eps_2}$, as otherwise with the probability at least $\frac{1}{3}$ both points $x_{k_4},y_{k_4}$ will be at the distance at least $\frac{\eps_3}{2}$ from its most contracted direction, and thus their images
\[
x_i= f_{T_{[k_4,i]}}(x_{k_4}), \quad y_i= f_{T_{[k_4,i]}}(y_{k_4})
\]
will be at the distance at most $\frac{1}{\|T_{[k_4,i]}\| \cdot \frac{\eps_3}{2}} < \frac{\eps_2}{2}$ from the image of the most expanded direction, and hence closer than $\eps_2$ to each other.

Thus, the products $A_{i}\dots A_{k_4+1}$ almost surely satisfy a uniform upper bound by~$\frac{4}{\eps_3 \eps_2}$. Hence, the norms of the matrices $A_i$ are also uniformly bounded.

Consider now the compact set $\bK'$ that is the closure of $\{\mgr_{k_4+1},\mgr_{k_4+2},\dots\}$. This is a closed subset of $\bK$, and hence a compact set of measures, on which the norms of matrices are uniformly bounded. Hence, we are again in the assumptions of the Nonstationary Furstenberg Theorem (\cite[Theorem~1.4]{GK2}). And thus the uniform bound on the norms of products $T_{[k_4,i]}$ provides us a contradiction, making this second case impossible.

This concludes the proof of Proposition~\ref{p.nF}.
\end{proof}
\begin{remark}
In case of a higher dimension, without additional assumptions (an analogue of absence of a finite invariant set of planes) the measure condition does not suffice to ensure the nonstationary ergodicity, as \cite[Example\,A.1] {GK2} shows. Nevertheless, with some extra assumptions (that would be sufficient to guarantee a non-stationary version of ``simplicity of Lyapunov spectrum'', or at least ``simplicity of the first Lyapunov exponent'') the conclusion of the proposition could be generalised to the case of $SL(d,\R)$, $d>2$.
\end{remark}

\begin{remark}
Contraction of random orbits for general stationary dynamics on the circle was established and studied in detail in \cite{A}, \cite{KN}, \cite{M}.  Even though in some parts of the proof of Proposition~\ref{p.nF} we have used the fact that we are composing projective maps, we expect that under some suitable assumptions the statement and the proof can be adapted to a more general case of non-stationary dynamics of circle homeomorphisms.
\end{remark}

\section{Proof of Theorem \ref{t.main}}

We will establish the second part of the theorem, the Large Deviations estimates, first; it implies the first part by an easy application of the Borel--Cantelli Lemma.

For a closed interval $J\subset \R$, denote by $\mM(J)$ the space  of Borel probability measures on~$J$ equipped with the Wasserstein metric.

Fix any $\varphi\in C(X, \mathbb{R})$ and set 
$M= \max(1, \max_{x\in X}|\varphi(x)|)$. The continuous map $\varphi:X\to \mathbb{R}$ induces a map on the space of Borel probability measures $\varphi_*:\mathcal{M}(X)\to \mathcal{M}([-M, M])$.
\begin{lemma}\label{l.contphi}
The map $\varphi_*:\mathcal{M}(X)\to \mathcal{M}([-M, M])$ is continuous.
\end{lemma}
\begin{proof}
Indeed, as $X$ is compact and thus $\varphi$ is uniformly continuous on $X$, there exists $\delta>0$ such that $d_X(x,y)<\delta$ implies $|\varphi(x)-\varphi(y)|<\frac{\eps'}{2}$. On the other hand, take $\eps'':=\frac{\eps' \delta}{4M}$. If for some coupling $\gamma$ between two measures $\msp, \msp'$ the integral~\eqref{eq:W-def} takes value less than $\eps''$, then by Markov inequality
\[
\gamma(\{(x,y) \mid d_X(x,y) > \delta \}) < \frac{\eps''}{\delta} = \frac{\eps'}{4M},
\]
and hence, using the pushforward $\gamma':=(\varphi,\varphi)_* \gamma$ as a coupling between $\varphi_* \msp$ and $\varphi_*\msp'$, we have
\begin{multline*}
\iint |u-v| \, d\gamma'(u,v) = \iint_{X\times X} |\varphi(x)-\varphi(y)| \, d\gamma(x,y) \le
\\
\le \frac{\eps'}{2} \cdot \gamma(\{(x,y) \mid d_X(x,y) \le \delta \}) + 2M \cdot \gamma(\{(x,y) \mid d_X(x,y) > \delta \} )
\\
< \frac{\eps'}{2} \cdot 1 + 2M \cdot \frac{\eps'}{4M} = \frac{\eps'}{2} + \frac{\eps'}{2} = \eps'.
\end{multline*}
Thus, the inequality $\dist_{\mM}(\msp,\msp')<\eps''$ implies $\dist_{\mM([-M,M])} (\varphi_*\msp,\varphi_*\msp')<\eps'$,
which proves Lemma \ref{l.contphi}.
\end{proof}

Since $\mathcal{M}(X)$ is compact, Lemma \ref{l.contphi} implies that $\varphi_*$ is uniformly continuous. Combined with the Standing Assumption this implies that for any $\eps'>0$ there exists $m$ such that for any $n$ and for any initial measure $\msp\in\mM$,
\begin{equation}\label{eq:d-phi}
\text{dist}_{\mM([-M,M])}(\varphi_*(\mu_{n+m}*\ldots *\mu_{n+1} *\msp), \varphi_* \msp_{n+m})<\eps',
\end{equation}
where the sequence of measures $\{\msp_i\}_{i\ge 0}$ is given by \eqref{eq:nu-def}.

Now, fix $\eps>0$, and let us obtain the Large Deviations estimates for this~$\eps$. Fix $m$ sufficiently large to make sure that \eqref{eq:d-phi} holds with $\eps' = \frac{\eps^2}{30M}$. 

Since $\varphi$ is bounded, it suffices to establish~(\ref{eq:LD}) for all $n\ge 1$ that are divisible by $m$. For any $n=mq$, $q\in \mathbb{N}$, decompose the iterations $k=1,\dots,mq$ into $m$ arithmetic sequences $k=r+jm$, $j=0,\dots, q-1$ with the difference $m$, indexed by the residue $r=1,\dots,m$.

It suffices to check that Large Deviations estimate holds for each of these subsequences. Indeed,
\begin{multline*}
 \frac{1}{n}\left|\sum_{k=1}^n \varphi(f_k\circ \ldots \circ f_1(x))-\sum_{k=1}^n \int_X \varphi d\nu_k \right|\le   \\
  \le \frac{1}{m}\sum_{r=1}^{m}\left|\frac{1}{q}\left(\sum_{j=0}^{q-1} \varphi(f_{jm+r}\circ \ldots \circ f_1(x))-\sum_{j=0}^{q-1} \int_X \varphi d\nu_{jm+r} \right)\right|.
\end{multline*}
Therefore, it suffices to prove that for each given $r=1, 2, \ldots, m$ we have for all $q\ge 1$ that
\begin{equation}\label{eq:m}
\mathbb{P}\left(\frac{1}{q}\left|\sum_{j=0}^{q-1} \varphi(f_{jm+r}\circ \ldots \circ f_1(x))-\sum_{j=0}^{q-1} \int_X \varphi d\nu_{jm+r} \right|>\varepsilon\right)<\frac{C}{m}e^{-\delta n},
\end{equation}
since that will imply (\ref{eq:LD}).

Fix $r\in\{1, 2, \ldots, m\}$, and denote
\[
y_j := x_{r+mj}= f_{r+mj}\circ\dots\circ f_1(x);
\]
then, the first sum in~\eqref{eq:m} is given by
$\sum_{j=0}^{q-1} \varphi(y_j).$

To establish the estimate~\eqref{eq:m}, we are going to compare this sum to a sum of \emph{independent} random variables $\xi_j$ that are distributed w.r.t. $\varphi_*\msp_{jm+r}$ respectively.
To do so, for the technical reasons of the construction, we will add two more random variables per iteration that are distributed uniformly on $[0,1]$ and that are independent from the iterations and altogether. Namely, we consider 
\[
\Omega:=\{(f_n,z_n,z'_n)_{n\in\N}\}=C(X,X)^{\N} \times [0,1]^\N \times [0,1]^\N,
\]
equipping it with the probability measure
\[
\tilde{\Prob}:=\Prob \times \Leb^{\N} \times \Leb^{\N},   
\]
where $\Prob$ is defined by \eqref{eq:Prob}.
\begin{lemma}\label{l:xi}
There exist random variables $\xi_j, \xi'_j$ on the space $(\Omega,\tilde{\Prob})$ such that:
\begin{itemize}
\item $\xi_j$ are independent for all $j$, and the law of $\xi_j$ is $\varphi_* \msp_{r+jm}$;
\item $\xi'_j$ are i.i.d Bernoulli variables, taking value~$1$ with the probability~$\frac{\eps}{8M}$.
\item For every $j$, if $\xi'_j$ takes value~$0$, then $|\varphi(y_j)-\xi_j|\le \frac{\eps}{3}$.
\end{itemize}
\end{lemma}
We postpone the (slightly technical) proof of this lemma until the end of this section. Note that it allows to complete the proof of Theorem~\ref{t.main}. Indeed, note first that the last conclusion of Lemma~\ref{l:xi} implies the inequality
\begin{equation}\label{eq:phi-xi}
|\varphi(y_j)-\xi_j|\le \frac{\eps}{3} + 2M \xi'_j;
\end{equation}
notice that here we are using the fact that $|\varphi|\le M$, and hence $|\varphi(y_j)-\xi_j|\le 2M$.

The sum in~\eqref{eq:m} can be estimated as
\begin{multline}\label{eq:triangle}
\frac{1}{q}\left|\sum_{j=0}^{q-1} \varphi(f_{jm+r}\circ \ldots \circ f_1(x))-\sum_{j=0}^{q-1} \int_X \varphi d\msp_{jm+r} \right| =
\\
=\frac{1}{q}\left|\sum_{j=0}^{q-1} \varphi(y_j)-\sum_{j=0}^{q-1} \E \xi_j \right| \le
\frac{1}{q}\sum_{j=0}^{q-1} \left| \varphi(y_j)- \xi_j \right|  + \frac{1}{q}\left| \sum_{j=0}^{q-1}  (\xi_j- \E\xi_j) \right|
\le
\\ \le
\frac{\eps}{3} + 2M \cdot \frac{1}{q} \sum_{j=0}^{q-1} \xi'_j + \frac{1}{q}\left| \sum_{j=0}^{q-1}  (\xi_j- \E\xi_j) \right|
\end{multline}
Now, note, that as $\E \xi'_j = \frac{\eps}{8M}<\frac{\eps}{6M}$, one has
\[
\tilde{\Prob}\left(\frac{1}{q} \sum_{j=0}^{q-1} \xi'_j > \frac{\eps}{6M}\right) < C_1 e^{- \delta_1 q}
\]
for some $\delta_1>0$ due to the Large Deviations estimate for the independent variables~$\xi'_j$. In the same way, as $\xi_j$ are independent and uniformly bounded by~$M$, from the Large Deviations estimate one has
\[
\tilde{\Prob}\left(\left| \frac{1}{q} \sum_{j=0}^{q-1} (\xi_j - \E \xi_j) \right| > \frac{\eps}{3}\right) <  C_2 e^{- \delta_2 q}
\]

Now, if both $\frac{1}{q} \sum_{j=0}^{q-1} \xi'_j\le \frac{\eps}{6M}$ and $\frac{1}{q}\left| \sum_{j=0}^{q-1}  (\xi_j- \E\xi_j) \right| < \frac{\eps}{3}$ hold,~\eqref{eq:triangle} implies
\[
\frac{1}{q}\left|\sum_{j=0}^{q-1} \varphi(y_j)-\sum_{j=0}^{q-1} \E \xi_j \right| \le \frac{\eps}{3} + 2M \cdot \frac{\eps}{6M} + \frac{\eps}{3} =\eps.
\]
Hence, the probability in~\eqref{eq:m} is bounded from above by the
\begin{multline*}
\Prob\left(\frac{1}{q}\left|\sum_{j=0}^{q-1} \varphi(y_j)-\sum_{j=0}^{q-1} \E \xi_j \right|>\eps\right)\le \\ \le \tilde{\Prob}\left(\frac{1}{q} \sum_{j=0}^{q-1} \xi'_j > \frac{\eps}{6M}\right)  + \tilde{\Prob}\left(\left| \frac{1}{q} \sum_{j=0}^{q-1} (\xi_j - \E \xi_j) \right| > \frac{\eps}{3}\right)
\\ < C_1 e^{- \delta_1 q}+   C_2 e^{- \delta_2 q},
\end{multline*}
and we obtain the desired exponentially small bound with $\delta=\min(\delta_1,\delta_2)$.

Since we have obtained the desired bound~\eqref{eq:m} that holds for each $r=1,\dots, m$, this completes the proof of Theorem~\ref{t.main} modulo Lemma~\ref{l:xi}.

\begin{proof}[Proof of Lemma~\ref{l:xi}]
On the space $(\Omega,\tilde{\Prob})$, define a sequence of $\sigma$-algebrae
\[
\mF_j:=\sigma(\{(f_k,z_k,z'_k) \mid k\le r+jm\}).
\]
Note that the random values $\varphi(y_i)=\varphi(x_{r+mi})$, $i=0,\dots, j-1$ are measurable w.r.t. $\mF_{j-1}$; meanwhile, the conditional distribution of $\varphi(y_{j})$ with respect to this $\sigma$-algebra is given by
\[
D_{j,y_{j-1}}:=\varphi_*(\mu_{r+mj}*\ldots *\mu_{r+1+m(j-1)} * \delta_{y_{j-1}}).
\]
Due to the choice of $m$, this conditional distribution is $\eps'$-close to~$\varphi_*\msp_{r+jm}$, that is the law of the desired random variable $\xi_j$. This suggests to construct $\xi_j$ and $\xi'_j$ in the following way. Assume that
\begin{enumerate}
\item\label{i:meas} for every $j$ the random variables $\xi_j$, $\xi'_j$ are measurable w.r.t. $\mF_j$;
\item\label{i:xi} the conditional distribution of $\xi_j$ w.r.t. $\{(f_k,z_k,z'_k) \mid k\le r+(j-1)m\}$ is~$\varphi_*\msp_{r+jm}$;
\item\label{i:xi-p} the conditional distribution of $\xi'_j$ w.r.t. $\{(f_k,z_k,z'_k) \mid k\le r+(j-1)m\}$ is Bernoulli with the success probability $\frac{\eps}{8M}$ ;
\item\label{i:diff} If $|\varphi(y_j)-\xi_j|> \frac{\eps}{3}$, then $\xi'_j$ takes value~$1$.
\end{enumerate}
Then these random variables satisfy the conditions of Lemma~\ref{l:xi}. Indeed, property~\eqref{i:xi} implies that $\xi_j$ is independent from $\sigma$-algebra $\mF_{j-1}$. As $\xi_0,\dots,\xi_{j-1}$ are measurable w.r.t. $\mF_{j-1}$ due to the property~\eqref{i:meas}, the random variable $\xi_j$ is independent from $\xi_0,\dots,\xi_{j-1}$.

In the same way, $\xi'_j$ is independent from $\xi'_0,\dots,\xi'_{j-1}$ due to~\eqref{i:meas} and~\eqref{i:xi-p}. The last conclusion of the lemma will follow from~\eqref{i:diff}.

Let us now construct random variables $\xi_j$, $\xi'_j$, satisfying~\eqref{i:meas}--\eqref{i:diff}.

Recall that between two probability measures on the real line, the optimal transport can be constructed explicitly. Namely, assume that we are given two probability measures $\rho, \rho'$ on the real line. Take two independent random variables $\eta,\zeta$, where $\eta$ is distributed w.r.t. $\rho$, and $\zeta$ is uniformly distributed on $[0,1]$. Let
\[
\Phi_\rho(y,z):=\rho((-\infty,y)) + z\cdot \rho(\{y\});
\]
then $\Phi(\eta,\zeta)$ is a uniformly distributed random variable (this is the standard principle of substituting a random variable to its own distribution function, where the independent random variable is used to ``dissolve the atoms''). On the other hand, an application of the function
\[
\Psi_{\rho'}(s) = \sup\{y: \rho'((-\infty,y))\le s\},
\]
that is an inverse to the distribution function of $\rho'$, sends the Lebesgue measure to~$\rho'$. Hence, their composition, $F_{(\rho,\rho')}:=\Psi_{\rho'}\circ \Phi_\rho$, provides a coupling between~$\rho$ and~$\rho'$. Namely, the random variable $\xi:=F_{(\rho,\rho')}(\eta,\zeta)$ is distributed w.r.t.~$\rho'$, and the desired coupling is given by the joint distribution~$(\eta,\xi)$. This coupling minimizes the Wasserstein distance \cite{V}.

Now, for each $j$ consider random variables
\[
\xi_j:= F_{(D_{j,y_{j-1}}, \varphi_*\msp_{jm+r})} (\varphi(y_j), z_{r+jm}).
\]
Properties~\eqref{i:meas} and~\eqref{i:xi} then are satisfied by construction.

Due to the optimality of this transport and the assumption on~$m$, conditionally to any $y_{j-1}$ we have
\[
\E (| \xi_j - \varphi(y_j) | \, \mid y_{j-1}) =\dist_{\mM([-M,M])}(D_{j,y_{j-1}}, \varphi_*\msp_{jm+r}) \le \eps'.
\]
Now, consider the event
\[
A_j:= \left\{|\xi_j - \varphi(y_j)| > \frac{\eps}{3}\right\}.
\]
Due to Markov inequality, conditionally to any $y_{j-1}$, its probability does not exceed
\[
\tilde{\Prob}(A_j \mid y_{j-1}) \le \frac{3}{\eps} \cdot  \E (| \xi_j - \varphi(y_j) | \, \mid y_{j-1}) \le \frac{3\eps'}{\eps} =\frac{3}{\eps}\cdot\frac{\eps^2}{30M}< \frac{\eps}{8M}.
\]

Now, define $\xi'_j$ by
\begin{equation}\label{eq:xi-p-1}
\xi'_j = \Ind_{A_j} + \Ind_{ |\xi_j - \varphi(y_j)| \le \frac{\eps}{3} } \cdot \Ind_{z_{r+jm} \le \psi(j,y_{j-1})},
\end{equation}
where the value
\[
\psi(j,y_{j-1}):=\frac{\frac{\eps}{8M}-\tilde{\Prob}(A_j\mid y_{j-1})}{1-\tilde{\Prob}(A_j \mid y_{j-1})} \in (0,1)
\]
is chosen in such a way that
\[
\tilde{\Prob}(\xi'_j  \mid y_{j-1}) = \frac{\eps}{8M}
\]
for any value of $y_{j-1}$. The choice~\eqref{eq:xi-p-1} ensures both property~\eqref{i:xi-p} (measurability follows from the explicit formulae) and~\eqref{i:diff} (that follows directly from the construction).

This completes the construction of random variables $\xi_j, \xi'_j$ and the verification of their properties, thus the proof of Lemma \ref{l:xi}.
\end{proof}

\section{Concluding remarks}
\subsection{Different phase spaces at different times}

{
A general remark is that in the non-stationary setting the phase spaces at different iterations are no longer canonically identified with each other. Namely, consider a random orbit
\[
x_{n}=f_{n}(x_{n-1}).
\]
An application of a time-dependent change of variables $y_n=h_n(x_n)$, where each $h_n$ is a homeomorphism of~$X$, sends an orbit of the initial system to
\[
y_n=g_n(y_{n-1}), \quad g_n=h_n \circ f_n \circ h_{n-1}^{-1}.
\]
This corresponds to the random iterations of another system from the same class, with measures $\mgr'_n=(T_n)_* \mgr_n$ that are push-forward images of $\mgr_n$ under the ``change of variable'' maps
\[
T_n: C(X,X)\to C(X,X), \quad T_n:f\mapsto h_n \circ f_n \circ h_{n-1}^{-1}.
\]
However, the equality between some $x_n$ and $x_m$ does not imply the equality between their images $y_n=h_n(x_n)$ and $y_m=h_m(x_m)$. This is why one can (and should) think of the spaces $X$ at different moments of time as of {\it different} fibers. 
}

{Moreover, in the absence of stationarity we do not need to assume that the phase space stays the same all the time. Rather, we can consider a sequence of compact metric spaces $X_0, X_1, X_2, \ldots$, and a sequence of probability measures $\mgr_n, n\ge 1,$ on $C(X_{n-1}, X_{n})$:
\[
X_0 \xrightarrow[f_1]{\mgr_1} X_1 \xrightarrow[f_2]{\mgr_2} X_2 \rightarrow \dots \rightarrow X_{n-1} \xrightarrow[f_n]{\mgr_n} X_{n} \rightarrow \dots.
\]
Then, given a sequence of ``observable'' functions $\varphi_n\in C(X_n, \mathbb{R})$, we can ask the same  questions as before. }

{It is reasonable to require a uniform bound on the diameters of $X_n$, and a uniform modulus of continuity for the functions $\varphi_n$; Theorem~\ref{t.main} then can be generalized to such a setting:}
\begin{theorem}\label{t:gen}
Assume that the diameters of $X_n$ are uniformly bounded, and that the Standing Assumption (with the replacement of $X$ by corresponding $X_n$'s), holds. Then for any sequence $\varphi_n\in C(X_n, \mathbb{R})$ of functions admitting a uniform modulus of continuity and any $x\in X_0$ we have that
$$
\frac{1}{n}\left|\sum_{k=1}^n \varphi_k(f_k\circ \ldots \circ f_1(x))-\sum_{k=1}^n \int_{X_k} \varphi_k d\msp_k \right|\to 0 \ \ \text{\rm as}\ \ \ n\to \infty,
$$
where {$\msp_{n}=\mgr_n*\msp_{n-1}$} for all $n\ge 1$, {and $\msp_0$ is any initial Borel probability measure on~$X_0$}.

Moreover, an analogue of the Large Deviations Theorem holds: for any $\eps>0$ there exists $C,\delta>0$ such that for any $x\in X_0$
\[
\forall n\in \mathbb{N},\ \quad \Prob\left(
\frac{1}{n}\left|\sum_{k=1}^n \varphi_k(f_{k}\circ \ldots \circ f_1(x))-\sum_{k=1}^n \int_{X_k} \varphi_k d\msp_k \right|> \eps
\right) < C \exp(-\delta n).
\]
\end{theorem}

The proof of Theorem \ref{t:gen} is almost verbatim repetition of the proof of Theorem \ref{t.main}, and as such will be omitted.

\subsection{Counter-examples: assumptions that cannot be avoided}\label{s:counter}

This section is devoted to presenting the examples where Nonstationary Random Ergodic Theorem does not hold. The first of these shows why one has to require at least some upper bound on the number of iterations needed to ``diffuse'' the initial measure in the Standing Assumption:
\begin{example}\label{ex:slow}
Let the set $X=\{a,b\}$ consist of two points only. Fix a (sufficiently quickly growing) sequence $n_k=10^{k^2}$, and let measure $\mu_n$ be defined as $\mu_n=\delta_{\id}$ for $n\neq n_k$ and
\[
\mu_{n_k} = \frac{1}{2} \delta_{\id} + \frac{1}{2} \delta_{\sigma},
\]
where $\sigma$ is a transposition that interchanges $a$ and~$b$. Then, on one hand, for any initial $n$, any two measures $\msp,\msp'$ on $X$, and any $n_k>n$, one has
\[
\mgr_{n_k}*\ldots *\mgr_{n+1} *\msp' = \mgr_{n_k}*\ldots *\mgr_{n+1} *\msp = \frac{1}{2} \delta_{a} + \frac{1}{2} \delta_{b} =:\overline{\msp}.
\]
On the other hand, any individual orbit $(x_n)$ spends each interval of time between~$n_k$ and~$n_{k+1}$ either fully at~$a$, or fully at~$b$, with each of these probabilities occurring equiprobably and independently for different~$k$. As the quotient $\frac{n_{k+1}-n_k}{n_k}$ tends to infinity, this easily implies that the time averages of any function $\varphi$ almost surely have both $\varphi(a)$ and $\varphi(b)$ as its accumulation points. In particular, these time averages (for non-constant $\varphi$) do not converge.
\end{example}

The next example shows why in the Standing Assumption we avoid averaging in time. Actually, the absence of the natural identification between the phase spaces at different times (discussed in the previous section) already suggests that time-average of images of a measure is not a good object to be considered: that would require an addition of measures on the different spaces. However, even if we are dealing with a uniquely ergodic deterministic dynamical system, so that such an addition can be considered, the Nonstationary Ergodic Theorem for $\varphi_n$ depending on $n$ without the Standing Assumption may not hold:

\begin{example}\label{ex:no-average}
Let $X=S^1=\R/\Z$, and $f(x)=x+\alpha \mod 1$ be an irrational rotation. Then, the (classical) dynamical system $(X,f)$ is uniquely ergodic (its unique invariant measure is Lebesgue measure); however, the Standing Assumption does not hold (there is no randomness, and the distances between the orbits do not decrease). At the same time,  Cesaro averages of the images of any two initial measures on the circle converge to Lebesgue measure with some uniform rate.

Now, take any function $\varphi\in C(X)$ and consider the family $\varphi_n=\varphi\circ f^{-n}$. This family of functions is equicontinuous on~$X$. At the same, the conclusion of Theorem~\ref{t:gen} does not hold for these functions: for any initial point~$x\in X$ one has
\[
\frac{1}{n}\sum_{k=1}^n \varphi_k(f_k\circ \ldots \circ f_1(x)) = \frac{1}{n}\sum_{k=1}^n \varphi\circ f^{-k}(f^k(x))
=  \varphi(x),
\]
so there is no constant-like behaviour. Also, replacing the choice of functions by
\[
\varphi_n = \varphi \circ f^{-n + r(n)},
\]
where $r(n):=\max\{k \mid n>n_k\}$ for a fast-growing sequence $n_k=10^{k^2}$, similar to Example~\ref{ex:slow}, one gets the absence of a limit of ``non-stationary'' time averages for every initial point~$x$.
\end{example}

\section*{Acknowledgments}

We are grateful to Andrew T\"or\"ok for his remarks on the preliminary draft of this paper.

\end{document}